\documentclass[letter,11pt]{article}
\newif\iffigures
\figurestrue 

\usepackage{amsmath,amsxtra} 
\usepackage{amssymb}  
\usepackage{color}
\usepackage{graphicx}
\usepackage{psfrag}
\usepackage{nicefrac}
\usepackage{fancyvrb}
\usepackage{url}
\usepackage{listings}
\usepackage[section]{placeins}
\usepackage{setspace}
\usepackage{cite}
\usepackage[T1]{fontenc}
\usepackage{aurical}
\usepackage{tikz,pgfplots,verbatim}
\usepackage{todonotes}
\usepackage{subcaption}
\usepackage{authblk}

\usetikzlibrary{arrows}

\iffigures
\usepgfplotslibrary{external} 
\makeatletter
\renewcommand{\todo}[2][]{\tikzexternaldisable\@todo[#1]{#2}\tikzexternalenable}
\makeatother
\fi

\definecolor{YEL}{RGB}{0, 255, 0} 
\definecolor{RED}{RGB}{185, 0, 0}
\definecolor{CYA}{RGB}{0, 194, 255}
\definecolor{BLU}{RGB}{0, 0, 128}
\definecolor{PIN}{RGB}{247, 127, 190}

\pgfplotsset{
tick label style={font=\small},
label style={font=\small},
legend style={font=\small},
grid style={dashed},}

\lstdefinestyle{compactCStyle}{%
mathescape=true,
frame=l,
numbers=left,
numberstyle=\tiny,
numbersep=5pt,
xleftmargin=15pt,
language=C,
morekeywords={min,boolean},
basicstyle={\footnotesize\ttfamily},
commentstyle=\rm,flexiblecolumns=true}

\newtheorem{definition}{Definition}[section]

\newtheorem{theorem}{Theorem}[section]
\newtheorem{proposition}{Proposition}[section]

\newenvironment{proof}{\textbf{Proof.}}{$\square$\\}

\newcommand{\tr}{^{\mathrm T}}

\newcommand{\magn}[1]{\left\vert #1 \right\vert}

\newcommand{\bitem}{\item[$\bullet$]}

\newcommand{\vv}{\mathbf{v}}
\newcommand{\tv}{\bar{v}}
\newcommand{\tvv}{\bar{\mathbf{v}}}
\newcommand{\cV}{\mathcal{V}}
\newcommand{\oVV}{\overline{\mathcal{V}}}
\newcommand{\vs}{\mathbf{s}}

\newcommand{\vz}{\mathbf{z}}
\newcommand{\df}{\doteq}

\newcommand{\RM}{$\mathsf{RM}$}
\newcommand{\NEG}[1]{\left[#1\right]_{-}}
\newcommand{\NEGs}[1]{[#1]_{-}}

\begin{document}

\title{Design and  Implementation of Distributed Resource Management for Time
    Sensitive Applications\thanks{The research leading to these results was supported by the Linneaus
    Center LCCC, the Swedish VR project n.2011-3635 ``Feedback-based
    resource management for embedded multicore platform'', and the
    Marie Curie Intra European Fellowship within the 7th European
    Community Framework Programme. An earlier version of parts of this
    paper appeared in \cite{Cha13} and its implementation framework
    appeared in \cite{MagECRTS}.}  }
    
    \author[1]{Georgios Chasparis\thanks{Corresponding Author: georgios.chasparis@scch.at}}
    \author[2]{Martina Maggio\thanks{martina.maggio@control.lth.se}}
    \author[2]{Enrico Bini\thanks{bini@control.lth.se}}
    \author[2]{Karl-Erik \AA rz\'en\thanks{karlerik@control.lth.se}}

   \affil[1]{{\small Software Competence Center Hagenberg GmbH, Softwarepark 21, A$-$4232 Hagenberg, Austria.}}
   \affil[2]{{\small Department of Automatic Control, Lund University, Box 118, SE$-$221 00 Lund, Sweden.}}

\date{June 4, 2013 \\ July 12, 2015 (revised)}

\maketitle

\begin{abstract}
In this paper, we address distributed convergence to \emph{fair} allocations of CPU resources for time-sensitive applications. We propose a novel resource management framework where a centralized objective for fair allocations is decomposed into a pair of performance-driven recursive processes for updating: (a) the allocation of computing bandwidth to the applications (\emph{resource adaptation}), executed by the resource manager, and (b) the service level of each application (\emph{service-level adaptation}), executed by each application independently. We provide conditions under which the distributed recursive scheme exhibits convergence to solutions of the centralized objective (i.e., fair allocations). Contrary to prior work on centralized optimization schemes, the proposed framework exhibits adaptivity and robustness to changes both in the number and nature of applications, while it assumes minimum information available to both applications and the resource manager. We finally validate our framework with simulations using the TrueTime toolbox in MATLAB/Simulink.
\end{abstract}

\section{Introduction}
The current trend in embedded computing demands that the number of applications sharing the same execution platform increases.  This is due to the increased capacity of the new hardware platforms, e.g., through the use of multi-core techniques. 
An example includes the move from federated to integrated system architectures in the automotive industry \cite{DiNatale10}. 

In such scenarios, the need for better mechanisms for controlling the rate of execution of each application becomes apparent. To this end, virtualization or resource reservation techniques~\cite{Mer94,Abe98a} are used. According to these techniques, each reservation is viewed as a \emph{virtual processor} (or \emph{platform}) executing at a fraction of the speed of the physical processor, i.e., the \emph{bandwidth} of the reservation. 
An orthogonal dimension along which the performance of an application can be tuned is the selection of its \emph{service level}. It is assumed that an application is able to execute at different service levels, where a higher service level implies a higher quality-of-service (QoS). Examples include the adjustable video resolutions and the adjustable sampling rates of a controller.

Typically this problem is solved by using a \emph{resource manager} (\RM), which is in charge of: (a) \textit{assigning virtual processors to the applications}, (b) \textit{monitoring the use of resources}, and (c) \textit{assigning the service level to each application}. The goal of the \RM\ is to maximize the overall delivered QoS.
This is often done through \textit{centralized optimization} and the
use of \textit{feedback} from the applications.

\RM's that are based on the concept of feedback, monitor the progress of the applications and adjust the virtual platforms based on measurements~\cite{Ste99,Eke00}. In these early approaches, however, quality adjustment was not considered. Instead, reference \cite{Cuc10} proposed an inner loop to control the resource allocation nested within an outer loop that controls the overall delivered quality.

Optimization-based resource managers have also received considerable attention~\cite{Raj97a,Lee99}. These approaches, however, rely on the
solution of a centralized optimization that determines \textit{both} the amount of assigned resources and the service levels
of all applications~\cite{Raj97a,Soj11,Bin11}. In the context of networking, reference~\cite{Joh06} models the service provided by a set of servers to workloads belonging to different classes as a utility maximization problem. However, there is no notion of adjustment of the service level of the applications.

An example of a combined use of optimization and feedback was developed in the ACTORS project~\cite{Bin11,Arz11}. In that project, applications provide a table to the \RM\ describing the required amount of CPU resources and the expected QoS achieved at each supported service level~\cite{Bin11,Arz11}.  In the multi-core case, applications are partitioned over the cores and the amount of resources is given for each individual partition. Then, the \RM\
decides the service level of all applications and how the partitions should be mapped to physical cores using a combination of Integer Linear Programming (ILP) and first-fit-decrease (FFD) for bin packing.

On-line centralized optimization schemes have several weaknesses. First, the complexity of the solvers used to implement the \RM\ (such as ILP solvers) grows significantly with the number of applications. It is impractical to have a \RM\ that optimally assigns resources at the price of a large consumption of resources by the \RM\
itself.  Second, to enable a meaningful formulation of a cost function in such optimization problems, the \RM\ must compare the quality delivered by different applications.  This comparison is unnatural because the concept of quality is extremely application dependent. Finally, a proper assignment of service levels requires application
knowledge. In particular, applications must inform the \RM\ about the available service levels and the expected consumed resources at each service level, increasing significantly communication complexity. 

To this end, distributed optimization schemes have recently attracted considerable attention. Reference~\cite{Sub08} considered a cooperative game formulation for job allocation to service providers in grid computing. Reference~\cite{Wei10} proposed a non-cooperative game-theoretic formulation to allocate computational resources to a given number of tasks in cloud computing. Tasks have full knowledge of the available resources and try to maximize their own utility function. Similarly, in \cite{Gro05} the load balancing problem is formulated as a non-cooperative game. 

Contrary to the grid computing setup of \cite{Sub08} or the load balancing problem of \cite{Gro05,Wei10}, this paper addresses a lower-level resource allocation problem, that is, \emph{the establishment of fair allocations of CPU bandwidth among time-sensitive applications which adjust their own service levels}. Contrary to the cloud computing setup of \cite{Wei10}, a game-theoretic formulation may not easily be motivated practically when addressing such lower-level (single node) resource allocation problems. Instead, we propose a distributed optimization scheme, according to which a centralized objective for fair allocations is decomposed into a pair of performance-driven recursive processes for updating: (a) the allocation of computing bandwidth to the applications (\emph{resource adaptation}), executed by the \RM, and (b) the service level of each application (\emph{service-level adaptation}), executed by each application independently. We provide conditions under which the distributed recursive scheme exhibits convergence to fair allocations). 

The proposed scheme introduces a design technique for allocating computing bandwidth to \textit{time-sensitive applications}, i.e., applications whose performance is subject to strict time deadlines, such as multimedia and control applications. In particular, the proposed scheme: (a) exhibits linear complexity with the number of applications, (b) drops the assumption that the \RM\ has knowledge of application details, and (c) exhibits adaptivity and robustness to the number and nature of applications. This paper extends the theoretical contributions of \cite{Cha13} by addressing global convergence and asynchronous updates. 
Furthermore, reference~\cite{MagECRTS} presents the full implementation framework in Linux. 

The paper is organized as follows. Section~\ref{sec:framework}
provides the overall framework, while Section~\ref{sec:LearningDynamics} presents the distributed scheme for resource allocation. Section~\ref{sec:convergence} presents the convergence behavior for the synchronous and asynchronous case. Section
\ref{sec:TechnicalDerivation} presents technical details required for
the derivation of the main results in Section~\ref{sec:convergence}. Section~\ref{sec:expsim} provides selective simulations. Finally, Section \ref{sec:conc}
presents concluding remarks.

\textit{Notation:} 
\begin{itemize}
\bitem $\Pi_{[a,b]}$ is the projection onto the set $[a,b]$.
\bitem For some finite sequence $\{x_1,x_2,...,x_n\}$ in $\mathbb{R}$,
  define ${\rm col}\{x_1,x_2,...,x_n\}$ to be the column vector in $\mathbb{R}^{n}$ with entries $\{x_1,x_2,...,x_n\}$.
\bitem For any $x\in\mathbb{R}$, define the operator $\NEG{x}$ as follows:
\begin{equation*}
\NEG{x} \triangleq \begin{cases}
x, & x\leq{0} \cr
0, & x > {0}.
\end{cases}
\end{equation*}
\bitem For any $x\in\mathbb{R}^{n}$ and set $A\subset{\mathbb{R}^{n}}$, define ${\rm dist}(x,A)\df\inf_{y\in{A}}\|x-y\|,$ where $\|\cdot\|$ denotes the Euclidean norm.
\bitem For some finite set $A$, $\magn{A}$ denotes the cardinality of $A$.
\end{itemize}

\section{Framework \& Problem Formulation} 
\label{sec:framework}

\subsection{Resource manager \& applications}
\label{sec:RMapp}

The overall framework is illustrated in Figure~\ref{fig:framework}. A
set $\mathcal{I}$ of $n$ (time-sensitive) applications are sharing the same CPU platform. Let $i$ be a representative element of this set. Since we allow applications to dynamically join or leave, the number $n$ may not be constant over time. 

The resources are managed by a \RM\ that allocates resources through a Constant Bandwidth Server (CBS)~\cite{Abe98a} with period $P_i$ and budget $Q_i$. Hence, application $i$ is assigned a \emph{virtual platform} with bandwidth $v_i=Q_i/P_i$ corresponding to a fraction of the computing power (or speed) of a single CPU. 
Obviously, not all virtual platforms $v_i$ are feasible, since their sum cannot exceed the number $\kappa$ of
available CPU's. Formally, we define the set of \textit{\textbf{feasible virtual platforms}}, $(v_1,\ldots,v_n)$, as
\begin{equation}
  \label{eq:feasibleVPs}
  \cV \df \Big\{\vv=(v_1,...,v_n)\in[0,1]^n:\sum_{i=1}^nv_i\leq \kappa\Big\}.
\end{equation}
\begin{figure}[t]
  \centering
  \includegraphics[width=.7\columnwidth]{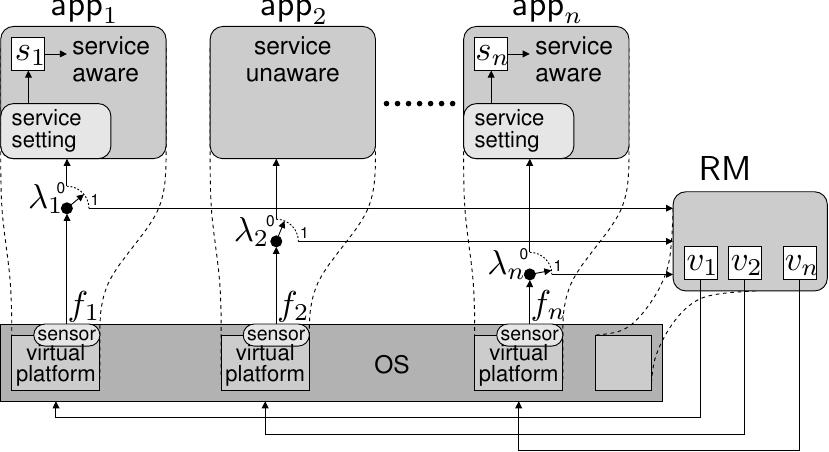}
  \caption{Resource management framework.}
  \label{fig:framework}
\end{figure}
%
%

In this study, the main concern is the computation of the allocation $\vv$ in real time such that a centralized objective is achieved. However, we will not be concerned with the exact mapping of this allocation onto the available cores. Such mapping can be performed by a standard first-fit-decrease algorithm. Furthermore, in practice, more constraints might be present, especially if applications are single-threaded (i.e., they may only run on a single core). In this case, the above feasibility constraint will be a relaxed version of the original problem, however, the forthcoming analysis can be modified in a straightforward manner to incorporate additional constraints on $\cV$.

Each application $i\in\mathcal{I}$ may change its \textit{service level}, $s_i$. It represents a qualitative indicator of the delivered quality of application $i$, assuming sufficient amount of resources $v_i$. Naturally, it can be represented by a real number $s_i\in\mathcal{S}_i\df[\underline{s}_i,\infty)\subset\mathbb{R}$, where $\underline{s}_i>0$ is the minimum possible service level of application $i$. The domain $\mathcal{S}_i$ inherits the partial ordering from $\mathbb{R}$, according to which $s_i'\leq s_i''$ implies that the quality delivered at service level $s_i'$ is smaller than or equal to the corresponding quality delivered at $s_i''$. The physical interpretation of the service level may only be realized in the context of a specific application. It may represent any quality indicator of the application, e.g., the inverse of the accuracy of an iterative optimization routine, the details of an MPEG player and the sampling frequency of a controller. We denote $\vs\df{\rm col}\{s_1,...,s_n\}$ the service level profile of all applications evolving within $\mathcal{S}\df \mathcal{S}_1\times...\times \mathcal{S}_n$. 

We implicitly assume here that an application may always increase its service level providing the necessary resources, however, in practice it will always be constrained due to the constraints imposed in $v_i$.  Note, finally, that the service level $s_i$ is an \textit{internal} state of application $i$, i.e., it can be written/read only by $i$.

\subsection{The matching function}
\label{sec:matchFi}

To be able to assess the performance of a time-sensitive application, it is necessary to introduce a \textit{performance function}. The \RM\ is able to measure at any time $t\geq{0}$, (a) the \emph{soft-deadline} of each application $i$, $D_i(t)$, which is the time duration of its last CPU reservation, and (b) the corresponding \emph{job-response time}, $R_i(t)$, which is the time elapsed from the start time to the finishing time of a job during its last reservation.  A natural definition of such performance function for time-sensitive applications is the following \textit{matching function}:
\begin{equation}
 f_i(t) \df \frac{D_i(t)}{R_i(t)}-1,
 \label{eq:matching}
\end{equation}
Note that $f_i \geq -1$, a property that will be used often.

Based on the above definition, we define a \textit{perfect} matching between $D_i$ and $R_i$ to be the situation at which $|f_i|\leq\delta$, for some small $\delta>0$. This is the case when application $i$ has the correct amount of resources. Instead, a \textit{scarce} matching describes a situation at which $f_i < -\delta$, i.e., when application $i$ does not have enough resources, and an \textit{abundant} matching describes a situation at which $f_i > \delta$, i.e., when application $i$ has more than enough resources.

\subsubsection{Nominal matching function}

The matching function depends indirectly on the virtual platform $v_i$ and the service level $s_i$ of application $i$. For a large class of applications, we may derive a nominal representation of the matching function, denoted $\varphi_i$, as a function of $s_i$ and $v_i$ as follows:
\begin{equation}
  \varphi_i(s_i,v_i) \df \beta_i\frac{v_i}{s_i}-1,
  \label{eq:fDefSimple}
\end{equation}
for some positive constant $\beta_i$. For example, for \textit{multimedia applications}, the soft deadline $D_i$ can be considered constant, while the response time can be defined as $R_i=\nicefrac{C_i}{v_i}$, where $C_i=\alpha_is_i$ is the execution time per job (at a service level $s_i$) and $v_i$ is the speed of execution.  Similarly, in \textit{control applications}, $R_i=\nicefrac{C_i}{v_i}$ where $C_i$ denotes nominal time of execution, while the soft deadline $D_i$ is considered inverse proportional to the sampling frequency (or service level) $s_i$, i.e.,
$D_i=\nicefrac{\alpha_i}{s_i}$. Both cases lead to a matching function with the form of (\ref{eq:fDefSimple}).

It is evident that the nominal matching function (\ref{eq:fDefSimple}) satisfies the following properties: For some
$s_i,s_i'\in\mathcal{S}_i$ and $v_i,v_i'\in\cV$: (P1) $s_i\neq 0 \Rightarrow \varphi_i(s_i,0) < 0$, that is, the
matching must certainly be scarce if no resources are assigned; (P2) $s_i \geq s_i'\Rightarrow \varphi_i(s_i,v_i) \leq \varphi_i(s_i',v_i)$, if application $i$ lowers its service level, then the performance should not decrease; (P3) $v_i \geq v_i'\Rightarrow \varphi_i(s_i,v_i) \geq \varphi_i(s_i,v_i')$, if the bandwidth of application $i$ decreases, then the performance should not increase.

\subsection{Application weights}
\label{sec:lambdaI}

The \RM\ may also assign weights to the applications. We introduce the weight $\lambda_i\in(0,1]$ to represent the \emph{importance} that the \RM\ assigns to application $i$ when adjusting its virtual platform $v_i$.  As we shall see in a forthcoming section, the weights $\{\lambda_i\}$ will determine the
direction of adjustment of the virtual platforms $\{v_i\}$ by the \RM. 
The weights $\{\lambda_i\}$ are considered given and determined by the \RM.

\subsection{Fair allocations \& objective}	\label{sec:FairAllocations}

To define \emph{fair allocations}, for each application $i$, we introduce the following \emph{\textbf{nominal fairness measure}}:
 \begin{eqnarray} \label{eq:FairAllocationMeasure} 
  \lefteqn{{\Phi}_i(\vs,\vv) \df}\cr && -(1-\tv_i)\lambda_i\NEG{\varphi_i(s_i,v_i)} + \tv_i \sum_{j\neq{i}}\lambda_j \NEG{\varphi_j(s_j,v_j)},
\end{eqnarray}
where $\tv_i\df\nicefrac{v_i}{\kappa}$ is the normalized virtual platform of $i$ over the number of cores.

The function $\Phi_i$ captures the deficiency in resources of application $i$ compared to the rest of applications. When application $i$ is not performing well, i.e., $\varphi_i(s_i,v_i)<0$, and its available resources $v_i$ are small, while the rest of applications are performing well, we should expect large values for $\Phi_i$. 
%
\begin{definition}[Fair allocation] \label{def:FairAllocation} 
For some service level profile $\vs\in\mathcal{S}$, a virtual platform profile $\vv^*\in\cV$ is fair or balanced if $\Phi_i(\vs,\vv^*)\equiv{0}$ for all $i\in\mathcal{I}$.
\end{definition}

According to Definition~\ref{def:FairAllocation}, an allocation of virtual platforms $\vv^*$ is \emph{fair} for application $i$ only if $v_i^*\neq{0}$, since at zero resources $\varphi_i(s_i,0)<0$ and $\Phi_i(\vs,\vv)<0$. Thus, an allocation $\vv^*$ is \emph{fair} if either (a) $\NEGs{\varphi_i(s_i,v_i^*)}\equiv{0}$ for all $i$ or (b) $\NEGs{\varphi_i(s_i,v_i^*)}<0$ for all $i$ and the ratio of resources $\nicefrac{\tv_i^*}{1-\tv_i^*}$ coincides with the corresponding ratio of weighted matching functions. Since $\NEGs{\varphi_i}\in[-1,0]$, case (b) implies that \textit{the resources are balanced with the negative performances}. For example, if $\tv_i^*$ is large compared to the rest $1-\tv_i^*$, then $\NEGs{\varphi_i}$ has to be sufficiently negative, i.e., application $i$ should not perform so well compared to the rest. Informally, there could not be application $i$ that monopolizes the resources at a fair allocation when $i$ performs well and the others do not. 


The above fairness definition introduces a potential centralized problem for fair allocations.
\begin{eqnarray}
\centering
\left\{\begin{array}{ll}
  \min_{\vs\in\mathcal{S},\vv\in\cV} & \sum_{i\in\mathcal{I}}\magn{\Phi_i(\vs,\vv)} \cr
  \mbox{s.t.} & \varphi_i(s_i,v_i) = 0, \quad \forall i\in\mathcal{I}.
\end{array}\right.
\end{eqnarray}
However, neither the \RM\ nor application $i$ has complete knowledge of the details of the nominal matching function $\varphi_i(s_i,v_i)$. Thus, on-line centralized optimization is highly prohibited. Instead, optimization may only be based upon measurements collected during run-time. 


\section{Adjustment Dynamics}
\label{sec:LearningDynamics}

In this section, the centralized objective of fair allocations is decomposed into a pair of performance-driven recursive schemes, executed independently by the \RM\ and the applications, thus avoiding the computation and communication complexity of centralized optimization.

\subsection{Resource adaptation}
\label{sec:LearningDynamicsRM}

The $\mathsf{RM}$ updates the bandwidth $\tv_i=\nicefrac{v_i}{\kappa}$, normalized with respect to the number of cores $\kappa$.  The \emph{unused bandwidth} is $v_r = \kappa-\sum_{i=1}^{n}v_i,$ and its normalized version is
$\tv_r = 1-\sum_{i=1}^{n}\tv_i$. 
At time instances $t_k$, $k=0,1,\ldots$ the \RM\ assigns resources as follows:
\begin{enumerate}
\item It measures the matching function $f_i = f_i(t_k)$ for each $i\in\mathcal{I}$, and computes $\NEGs{f_i(t_k)}$.
\item It updates the normalized resource allocation vector $\tvv \df (\tv_1,...,\tv_n)$ as follows:
  \begin{equation}	\label{eq:RecursionForResources}
    \tv_i(t_{k+1}) = \Pi_{\oVV_i}\Big[\tv_i(t_k) + \epsilon F_i(t_k) \Big]
  \end{equation}
  for each $i=1,...,n$, where $\oVV_i\df[0,\nicefrac{1}{\kappa}]$ and $F_i(t_k)$ is the \emph{observed} fairness measure defined as follows: 
  \begin{eqnarray*}  
  \lefteqn{F_i(t_k) \df } \cr &&  -(1-\tv_i(t_k))\lambda_i \NEGs{f_i(t_k)} + \tv_i(t_k) \sum_{j\neq{i}}\lambda_j \NEGs{f_j(t_k)}.
  \end{eqnarray*}
  Furthermore, the unused bandwidth is updated according to $\tv_r(t_{k+1}) = 1-\sum_{i=1}^{n}\tv_i(t_{k+1})$.
\item It computes the original bandwidths by setting $v_i(t_{k+1}) = \kappa\,\tv_i(t_{k+1})$.
\item It updates the time index $k\leftarrow {k+1}$ and repeats.
\end{enumerate}

Note that according to the definition of $F_i(t_k)$, if there is a deficiency of resources for $i$, i.e., $F_i(t_k)>0$, then $\tv_i$ will increase, otherwise it will decrease. We consider a \emph{constant} step size $\epsilon>0$, since it provides an adaptive response to changes in the number of applications.
In some cases, we will use vector notation, denoting $\tvv\df {\rm col}\{\tv_1,...,\tv_n\}$ which evolves over $\oVV\df\oVV_1\times...\times\oVV_n$.


Recursion (\ref{eq:RecursionForResources}) for the adjustment of resources was motivated by the standard \textit{replicator dynamics} (cf.,~\cite[Chapter~3]{Weibull97}) and in particular the discrete-time equivalent (namely \textit{reinforcement learning}) introduced in \cite{ChasparisShamma11_DGA}.
Note that the \RM\ time complexity is linear with respect to the number of applications, as demonstrated in \cite{MagECRTS}.

\subsection{Service level adaptation}
\label{sec:AppAdjust}

The \RM\ provides information to each application $i$ through an observation signal $Y_i(t_k)$, $k=0,1,...$, that captures its performance. Applications are designed to adjust their service levels based on $Y_i(t_k)$ as follows: 
\begin{equation}	
\label{eq:RecursionForApplications}
 s_i(t_{k+1}) = \Pi_{\mathcal{S}_i}\left[s_i(t_k) + \epsilon Y_{i}(t_k)\right], \quad i\in\mathcal{I}.
\end{equation}
A natural selection for the observation signal is to set $Y_i(t_k)\equiv f_i(t_k)$, i.e., the observed matching function. In this scenario, the application $i$ will increase its service level if $f_i(t_k) > 0$, otherwise it will decrease it. Alternative observation terms can also be defined with similar properties as demonstrated in \cite{Cha13}.

\section{Convergence}
\label{sec:convergence}

In this section, a characterization of the convergence properties of the proposed distributed scheme is provided in case of (a) \emph{synchronous} applications' updates, and (b) \emph{asynchronous} applications' updates. 
Asynchronous updates constitute a form of perturbation of the nominal synchronous behavior which may alter significantly the performance of the scheme. Perturbations due to measurement noise are not present, since the \RM\ has direct access to the response time of each application. However, internal uncertainties of an application may result in small deviations from its nominal matching function. Due to the small probability density of such events, we will not discuss robustness with respect to such uncertainties, i.e., \textit{\textbf{for the remainder of the paper}, we consider $f_i(t) \equiv \varphi(s_i(t),v_i(t))$, where the nominal matching function satisfies (\ref{eq:fDefSimple}).}

\subsection{Feasibility}
\label{sec:Feasibility}

The first property of the proposed adjustment process is the feasibility of the resulting virtual platforms. 
\begin{proposition}[Feasible allocations] \label{Pr:FeasibleAllocations} For sufficiently small step size $\epsilon=\epsilon(n)>0$, the update recursion of projected virtual platforms (\ref{eq:RecursionForResources}) leads to a sequence of virtual platforms $\{\vv(t_k)\}$ which satisfies  $\vv(t_k)\in\cV$ for all $k=0,1,...$ as long as $\vv(t_0)\in\cV$.
\end{proposition}
\begin{proof}
  See Appendix~\ref{Ap:FeasibleAllocations}.
\end{proof}

\subsection{Minimum guarantees}
\label{sec:MinimumGuarantees}

The adjustment process guarantees \emph{starvation avoidance}, i.e., a positive amount of resources (at least $\epsilon>0$) to all applications with non-zero weight. Furthermore, it guarantees a \emph{balance} condition, according to which, in overloaded CPU's, no application is able to monopolize resources.

Before stating formally these observations, define: 
\begin{itemize}
\item $L \df \sup_{i\in\mathcal{I},k\in\mathbb{N}}|F_i(t_k)|<\infty$, 
\item $\lambda\df\min_{i\in\mathcal{I}}\lambda_i > 0$.
\end{itemize}

\begin{proposition}[Starvation avoidance] \label{Pr:StarvationAvoidance}
There exists $\epsilon^*=\epsilon^*(n)<\nicefrac{1}{(L+1)\kappa}$ with $\epsilon^*\to{0}$ as $n\to\infty$, such that for any step size $\epsilon\leq\epsilon^*$,  $\inf_{k\in\mathbb{N}}\tv_i(t_k)>\epsilon$ for all $i$.
\end{proposition}
\begin{proof}
See Appendix~\ref{Ap:StarvationAvoidance}.
\end{proof}
%

\begin{proposition}[Balance] \label{Pr:Balance}
Pick $0<\zeta\leq\nicefrac{1}{\kappa}$ such that $\max_{i\in\mathcal{I}}\{\beta_i\kappa\zeta/\underline{s}_i-1\} < 0$. For any $\epsilon=\epsilon(\zeta)<\zeta/L$, there exists a number of applications $n^*=n^*(\zeta)$ such that, for any set of applications $\mathcal{I}$ with $\magn{\mathcal{I}}\geq{n^*}$ and for any $i\in\mathcal{I}$, the following hold: 
\begin{enumerate}
\item if $\tv_i(t_0)>\zeta$, then $\tv_i(t_k)\leq \zeta$ after a finite $k$; 
\item if $\tv_i(t_0)\leq\zeta$, then $\tv_i(t_k)\leq\zeta$ for all $k=1,2,...$.
\end{enumerate}
Also, as $\zeta\to{0}$, $n^*(\zeta)\to \infty$ and $\zeta n^*(\zeta)\to{c}$, for some positive constant $c$.
\end{proposition}
Proposition~\ref{Pr:Balance} states that if we pick $\zeta$ such that $\tv_i\leq\zeta$ implies negative matching function for all $i$, and we consider a sufficiently large number of applications $n\geq{n^*}$, then all applications will end up with a virtual platform less than $\zeta$ within finite time. Informally, when the CPU is overloaded, no application can monopolize the available resources.
\begin{proof}
See Appendix~\ref{Ap:Balance}.
\end{proof}



\subsection{Synchronous convergence}		\label{sec:SynchronousConvergence}

In the forthcoming convergence analysis, we will consider either one of the following hypotheses:
\begin{enumerate}
\item[(H1)] Let $\beta_i/\underline{s}_i < 1$ for all $i$.
\item[(H2)] Let the number of applications $n$ be sufficiently large such that, there exists $0<\zeta\leq\nicefrac{1}{\kappa}$ satisfying properties (1) and (2) of Proposition~\ref{Pr:Balance} for $n^*(\zeta)\leq{n}$.
\end{enumerate}
Hypothesis (H1) corresponds to the case where the applications are highly demanding, while (H2) corresponds to the case where the assigned resources is small compared to the number of applications. 

The asymptotic behavior of recursions (\ref{eq:RecursionForResources})--(\ref{eq:RecursionForApplications}) can be associated with the limit points of the following collection of (nonlinear) ordinary differential equations (ODE):
\begin{equation} \label{eq:OverallODE}
 \left(\begin{array}{c}
    \dot{s}_i(\tau) \\ \dot{\tv}_i(\tau)
  \end{array}
    \right) = \left(\begin{array}{c}\varphi_i(s_i(\tau),\kappa\tv_i(\tau)) \\ \Phi_i(\vs(\tau),\kappa\tvv(\tau)) \end{array}\right) + \vz_i(\tau), \quad i\in\mathcal{I},
\end{equation}
as the step size $\epsilon$ approaches zero, where $\tau$ refers to the time-index of the ODE.
The vector $\vz_i(\tau)$ represents the vector of minimum length required to drive $\tv_i(\tau)$ back to $\oVV_i$ and $s_i(\tau)$ back to $\mathcal{S}_i$. Define $(\vs^{\tau_0}(\cdot),\tvv^{\tau_0}(\cdot))$ to be the solution of the ODE~(\ref{eq:OverallODE}) starting at $(\vs(\tau_0),\tvv(\tau_0))$. 

Consider also the linear-time interpolation $(s_{i,\epsilon}(t),\tv_{i,\epsilon}(t))$ of $\{(s_{i}(t_k),\tv_i(t_k))\}_k$, defined as follows: $s_{i,\epsilon}(t) = s_i(t_k),$ and $\tv_{i,\epsilon}(t) = \tv_i(t_k)$, for every $t_k \leq t < t_{k+1}$. Introduce also the vector notation $\vs_{\epsilon}\df{\rm col}\{s_{i,\epsilon}\}_i$ and $\tvv_{\epsilon}\df{\rm col}\{\tv_{i,\epsilon}\}_i$. 

\begin{theorem}[Synchronous convergence] \label{Th:SynchronousConvergence}
The following hold:
\begin{enumerate}
\item If either (H1) or (H2) applies, the ODE~(\ref{eq:OverallODE}) exhibits stationary points $(\vs^*,\tvv^*)$, which satisfy: 
\begin{eqnarray}	\label{eq:StationaryPoints}
s_i^*=\underline{s}_i, \mbox{ and } \begin{cases}\Phi_i(\vs^*,\kappa\tvv^*)=0, \ \ \mbox{or} \\ \Phi_i(\vs^*,\kappa\tvv^*)>0,\tv_i^* = \nicefrac{1}{\kappa} \end{cases} \forall {i}.
\end{eqnarray}
\item If either (i) $\beta_i\kappa/\underline{s}_i\to{0}$ for all $i$, or (ii) $n\to\infty$, then
\begin{itemize}
\item[(a)] any stationary point of the ODE~(\ref{eq:OverallODE}) satisfies:
\begin{eqnarray} \label{eq:SpecialStationaryPoint} 
s_i^* = \underline{s}_i, \ \ \tv_i^* \to  \min\Big\{ \frac{1}\kappa, \frac{\lambda_i}{\sum_j\lambda_j}\Big\}, \quad \forall i\in\mathcal{I}.
\end{eqnarray}
\item[(b)] $(s_i(t_k),\tv_i(t_k))\to(s_i^*,\tv_i^*)$ as $k\to\infty$ and $\epsilon\to{0}$.\footnote{By $x(t)\to{A}$ for a set $A$, we mean $\lim_{t\to\infty}{\rm dist}(x(t),A)=0$.}
\end{itemize}
\end{enumerate}
\end{theorem}
\begin{proof}
  The proof is an immediate implication of a series of propositions presented in detail in Section~\ref{sec:TD:SynchronousConvergence}.
\end{proof}

In other words, Theorem~\ref{Th:SynchronousConvergence} states that stationary points of the ODE~(\ref{eq:OverallODE}) are \emph{fair allocations} (except for trivial cases where a virtual platform is limited by the size of the core). Furthermore, when the CPU is overloaded (i.e., either due to (i) a high demand, or (ii) a large number of applications), then \emph{the unique fair allocation is a global attractor of the distributed process}.

\subsection{Asynchronous convergence}		\label{sec:AsynchronousConvergence}

So far, we have implicitly assumed that the adjustment dynamics (\ref{eq:RecursionForResources})--(\ref{eq:RecursionForApplications}) have synchronized clocks. However, the virtual platform, $v_i$, indirectly determines application $i$'s speed of execution. Hence, the update rate of the service level $s_i$ \emph{varies over time}.

Under asynchronous updates, the asymptotic allocation of virtual platforms may not necessarily be fair to all applications. Consider, for example, the case where an application $i$ does \textit{not} update its service level, while all other applications do. Then, under limited available resources, application $i$ will retain a
sufficiently negative matching function $f_i$, while the matching functions of all other applications steadily approach zero. This situation may lead to application $i$ getting asymptotically a larger virtual platform independently of its weight $\lambda_i$.

To address asynchronous updates, we first introduce the following notation, also visualized in
Figure~\ref{fig:Asynchronicity}.
\begin{itemize}
\bitem $t$ denotes the actual run time;
\bitem $t_k^i$ denotes the update instances of application $i$;
\bitem $t_m$ denotes the update instances of the \RM ;
\bitem $\bar{k}(t,i) \df \{k\in\mathbb{N}: t_k^i\leq t < t_{k+1}^i\}$ denotes the most recent to $t$ update index of application $i$;
\bitem $\bar{m}(t) \df \{m\in\mathbb{N}: t_m\leq{t} < t_{m+1}\}$ denotes the most recent to $t$ update index of the \RM ;
\bitem $\psi_i(m) \df \max\{m'\leq{m}: \exists ~ k \mbox{ s.t. } t_{m'}\leq{t_k^i}<t_{m}\}$ denotes the most recent update of the \RM\ after which the last update of application $i$ occurred. For example, in Figure~\ref{fig:Asynchronicity}, $\psi_i(\bar{m}(t)) = \bar{m}(t_{\bar{k}}^{i})$. We set $\psi_i(m)=0$ if there exists no ${k}$ such that $t_k^i<t_m$.  
\bitem $N_i(k) \df \bar{m}(t_{k+1}^i)-\bar{m}(t_k^i)$ is the number of times that the \RM\ has updated within $[t^i_k,t^i_{k+1})$, i.e., between two consecutive updates of application $i\in\mathcal{I}$.
\end{itemize}

\begin{figure}[t!]
  \centering
  \includegraphics[scale=1]{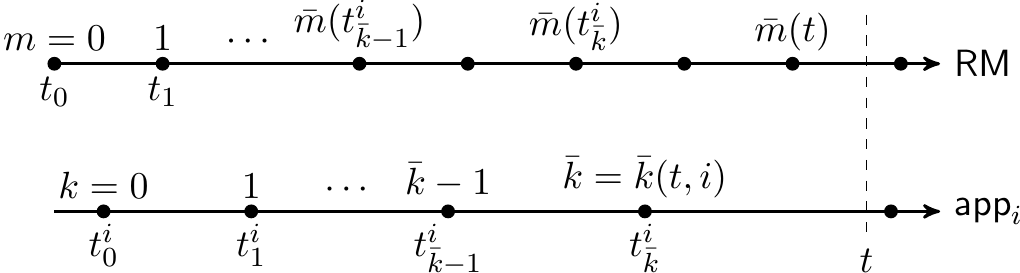}
  \caption{Visualization of asynchronous applications.}
  \label{fig:Asynchronicity}
\end{figure}

Throughout this section, we also admit the \textit{design assumption} that $1\leq N_i(k)\leq\bar{N}$ for all $i\in\mathcal{I}$ and $k=0,1,...$, and for some $\bar{N}\in\mathbb{N}$. In other words, (a) the frequency at which the \RM\ updates is larger than the frequency of every application $i$, and (b) each application $i$ updates with positive frequency. Without loss of generality, we will also assume that $t_0=0$, i.e., the \RM\ starts updating first.


Given the above notation, the update recursion of each application $i$
can be written as follows:
\begin{equation} \label{eq:ActualRecursion} 
  s_i(t_{k+1}^i) 
  = \Pi_{\mathcal{S}_i}[s_i(t_{k}^i) + \epsilon Y_{i}(t_{\bar{m}(t_k^i)})],
\end{equation}
since the computation of the observation signal $Y_{i}(\cdot)$ of each application $i$ is based upon the most recent performance measurements taken by the \RM\ at $\bar{m}(t_k^i)$.

\begin{theorem}[Asynchronous convergence] \label{Th:AsynchronousConvergence} 
Let application $i$'s actual recursion be given by (\ref{eq:ActualRecursion}), with
  \begin{equation} \label{eq:AsynchronicityCondition} 
  Y_{i}(t^i_k) = N_i(k)\cdot Y_{i}'(t^i_k),
  \end{equation}
  for each $i\in\mathcal{I}$, where $Y_{i}'(\cdot) \df {f}_i(\cdot)$. Then, the conclusions of Theorem~\ref{Th:SynchronousConvergence} continue to hold.
\end{theorem}
\begin{proof}
  The proof will be shown in Section~\ref{sec:TD:AsynchronousConvergence}.
\end{proof}

\section{Technical Derivation} 
\label{sec:TechnicalDerivation}

In this section, the technical proofs of
Theorems~\ref{Th:SynchronousConvergence}--\ref{Th:AsynchronousConvergence}
are presented through a series of propositions.

\subsection{Synchronous convergence (Theorem~\ref{Th:SynchronousConvergence})}	
\label{sec:TD:SynchronousConvergence}

The proof of Theorem~\ref{Th:SynchronousConvergence} is an immediate implication of the following steps: (a) derivation of the \textit{ODE approximation} of the adjustment dynamics (\ref{eq:RecursionForResources})--(\ref{eq:RecursionForApplications}), provided by Proposition~\ref{Pr:LAS}, (b) characterization and stability analysis of its stationary points, provided by Propositions~\ref{Pr:StationaryPoints}--\ref{Pr:GAS}.

\subsubsection{ODE approximation}	
\label{sec:TD:SC:ODE}

We begin by establishing a connection between the asymptotic properties of the recursions (\ref{eq:RecursionForResources})--(\ref{eq:RecursionForApplications}) with the locally asymptotically stable sets of the ODE~(\ref{eq:OverallODE}).

\begin{proposition}[Synchronous ODE]		\label{Pr:LAS}
Consider the recursions (\ref{eq:RecursionForResources})--(\ref{eq:RecursionForApplications}), according to which both the \RM\ and the applications update synchronously at fixed time instances $t_k$, $k=1,2,...$. If $A$ is a locally asymptotically stable set in the sense of Lyapunov\footnote{See \cite[Definition~3.1]{Khalil92}.} for the ODE~(\ref{eq:OverallODE}), then, for any initial condition $(\vs(t_0),\tvv(t_0))$ in the domain of attraction of $A$, $(\vs(t_k),\tvv(t_k))\to{A}$ as $k\to\infty$ and $\epsilon\to{0}$.
\end{proposition}
\begin{proof}
  The observation signal of the overall recursion is uniformly bounded, and the vector field of the ODE~(\ref{eq:OverallODE}) is a continuous function on its domain. By \cite[Theorem~1.1]{Iserles09} (which shows convergence of Euler's method), we have that for every $\tau>0$:
  \begin{eqnarray*}
  \lim_{\epsilon\to{0}}\sup_{k=0,...,\lfloor{\tau/\epsilon}\rfloor}\|(\vs_{\epsilon}(t_k),\tvv_{\epsilon}(t_k)) - (\vs^{\tau_0}(\tau_k),\tvv^{\tau_0}(\tau_k))\| = 0,
  \end{eqnarray*}
  where $(\vs^{\tau_0}(\tau_0),\tvv^{\tau_0}(\tau_0)) = (\vs(t_0),\tvv(t_0))$ and $\tau_k\df\epsilon{k}$. Given that $A$ is locally asymptotically stable and the initial condition $(\vs(t_0),\tvv(t_0))$ lies within the region of attraction of $A$, the conclusion follows in a straightforward manner by the fact that any solution of the ODE~(\ref{eq:OverallODE}) with initial condition $(\vs(t_0),\tvv(t_0))$ converges to $A$.
\end{proof}


\subsubsection{Stationary points}
\label{sec:TD:SC:StationaryPoints}

In this section, we characterize the stationary points of the ODE~(\ref{eq:OverallODE}). In general, if an allocation $(\vs^*,\tvv^*)$ exists such that $\phi_i(s_i^*,\kappa\tv_i^*)\equiv{0}$ for all $i$, then such allocation will be a stationary point of the ODE~(\ref{eq:OverallODE}) and a \textit{fair allocation} according to Definition~\ref{def:FairAllocation}. In situations though where the CPU is overloaded, there might not be such allocations. The following proposition provides a characterization of the stationary points in such cases.

\begin{proposition}[Stationary points]  \label{Pr:StationaryPoints}
Under either (H1) or (H2), the ODE~(\ref{eq:OverallODE}) exhibits stationary points satisfying  (\ref{eq:StationaryPoints}). Furthermore, as either (i) $\beta_i/\underline{s}_i\to{0}$ or (ii) $n\to\infty$, any stationary point satisfies (\ref{eq:SpecialStationaryPoint}).
\end{proposition}
\begin{proof}
  If hypothesis (H1) is satisfied, then $\phi_i(s_i,\kappa\tv_i)<0$ for all $(\vs,\tvv)$ and $\tau\geq{0}$. In this case, any stationary point $(\vs^*,\tvv^*)$ satisfies (\ref{eq:StationaryPoints}) which equivalently implies that: $s_i^* \equiv \underline{s}_i$ and 
  \begin{equation} \label{eq:StationaryPoints:Condition}
  \tv_i^* = \min\Big\{\frac{1}{\kappa}, \frac{\lambda_i\phi_i(s_i^*,\kappa\tv_i^*)}{\sum_{j\in\mathcal{I}}\lambda_j\phi_j(s_j^*,\kappa\tv_i^*)}\Big\}.
  \end{equation}
The mapping defined by the second expression of $\tv_i^*$ is well defined since $\phi_i(s_i^*,\kappa\tv_i^*)<0$ for all $i$. Furthermore, according to Brower's fixed point theorem \cite[Corollary~6.6]{Border85}, it exhibits at least one fixed point since it is a continuous mapping defined on a compact set. The possibility that $\tv_i^*=0$ is excluded by Proposition~\ref{Pr:StarvationAvoidance}. Finally, under hypothesis (H1), if we take $\beta_i/\underline{s}_i \to 0$ for all $i$, then $\phi_i(s_i,v_i)\to{-1}$ for all $i$, which further implies property (\ref{eq:SpecialStationaryPoint}).
  
If, instead, hypothesis (H2) is satisfied, then, by Proposition~\ref{Pr:Balance}, there exists a finite $k^*$, such that, $\tv_i(t_k) \leq \zeta$ for all $k>k^*$. By convergence of Euler's method, $\phi_i(s_i^{\tau_0}(\tau),\kappa\tv_i^{\tau_0}(\tau))<0$ for all $\tau\geq\tau_{k^*}\df\epsilon{k^*}$ and all $i$. Thus, the fixed-point property (\ref{eq:StationaryPoints:Condition}) also applies. Furthermore, if $n\to{\infty}$, then by Proposition~(\ref{Pr:Balance}), $\zeta\to{0}$, and $\phi_i(s_i^{\tau_0}(\tau),\kappa\tv_i^{\tau_0}(\tau))\to{-1}$ uniformly on $\tau\geq\tau_{k^*}$ and $i$, which implies (\ref{eq:SpecialStationaryPoint}).
\end{proof}


\subsubsection{Global Asymptotic Stability (GAS)}

The following proposition provides a strong convergence property of the ODE~(\ref{eq:OverallODE}).
\begin{proposition}[GAS]	
\label{Pr:GAS}
If either (i) $\beta_i/\underline{s}_i\to{0}$ or (ii) $n\to\infty$, then the unique stationary point of the ODE~(\ref{eq:OverallODE}), satisfying property (\ref{eq:SpecialStationaryPoint}), is globally asymptotically stable in the sense of Lyapunov.
\end{proposition}
\begin{proof}
See Appendix~\ref{Ap:GAS}.
\end{proof}

%

\subsection{Asynchronous convergence (Theorem~\ref{Th:AsynchronousConvergence})}	
\label{sec:TD:AsynchronousConvergence}

The proof of Theorem~\ref{Th:AsynchronousConvergence} is a direct implication of establishing \emph{equivalence} between the synchronous and asynchronous update recursions satisfying property (\ref{eq:AsynchronicityCondition}). In particular, we define \emph{equivalence} between two (deterministic) update recursions as follows. 


\begin{definition}[Equivalent updates]	\label{def:EquivalentUpdates}
Two update recursions of the form (\ref{eq:ActualRecursion}) and observation signals $\{Y_{i}(t_k^i)\}$ and $\{Y_i'(t_{k}^{i})\}$, $i\in\mathcal{I}$, are equivalent if the corresponding linear-time interpolations of the updated variables, $s_{i,\epsilon}(\cdot)$ and $s_{i,\epsilon}'(\cdot)$, respectively, satisfy $$\lim_{\epsilon\to{0}}\sup_{t\geq{0}}\magn{s_{i,\epsilon}(t) - s_{i,\epsilon}'(t)} = 0.$$
\end{definition}
In other words, two deterministic update recursions of the form (\ref{eq:ActualRecursion}) are considered equivalent when the corresponding linear-time interpolations approach each other uniformly in time as $\epsilon$ approaches zero. 



We introduce the following \textit{fictitious} recursion for each $i$,
\begin{equation}	\label{eq:FictitiousUpdate}
s_i'(t_{m+1}) = \Pi_{\mathcal{S}_i}\left[s_i'(t_m) + \epsilon Y_{i}'(t_{\psi_i(m)})\right],
\end{equation}
for all $m\ge{0}$. This update is synchronized with the time index of the \RM\ and $Y_{i}'(\cdot) \df f_i(\cdot)$. Note that the fictitious observation signals are defined at times $\{t_{\psi_i(m)}\}$, i.e., at the most recent update of the \RM\ prior to the last update of application $i$. Since the \RM\ starts updating first, we also set $Y'_{i}(t_{\psi_i(m)})\equiv{0}$ for all $m$ such that $\psi_i(m)=0$, since no performance measurements are available at time $t=0$.

The following proposition shows that, if we pick appropriately the observation signals of the (actual) asynchronous update (\ref{eq:ActualRecursion}), then the asynchronous update becomes equivalent with the synchronous update of (\ref{eq:RecursionForApplications}).

\begin{proposition}[Equivalence] \label{Pr:FictitiousEquivalence} 
  For each application $i\in\mathcal{I}$, assume that its actual update recursion is given by (\ref{eq:ActualRecursion}), where $Y_{i}(t^i_k)$ satisfies property (\ref{eq:AsynchronicityCondition}). 
  Then, the following statements hold:
  \begin{enumerate}
  \item The fictitious synchronous update (\ref{eq:FictitiousUpdate}) is equivalent with the asynchronous update (\ref{eq:ActualRecursion}).
  \item The fictitious synchronous update (\ref{eq:FictitiousUpdate}) is equivalent with the synchronous update (\ref{eq:RecursionForApplications}).
  \end{enumerate}
\end{proposition}
\begin{proof}
  See Appendix~\ref{Ap:FictitiousEquivalence}.
\end{proof}

%

\section{Experimental Evaluation}
\label{sec:expsim}

\subsection{Simulation framework}
\label{sec:implementation}

To test the assignment of virtual platforms $v_i$ and service levels $s_i$, the resource management framework was implemented in TrueTime~\cite{cer03}. TrueTime is a MATLAB/Simulink-based tool, embedded within Simulink, that allows the simulation of tasks executing within real-time kernels.
TrueTime implements virtual processors through the Constant Bandwidth Server (CBS)~\cite{Abe98a}. Also, it is possible to adjust the CPU time allocated to the running applications (the bandwidth $v_i$), as needed by our \RM.
Moreover, TrueTime offers the ability to simulate memory management and protection, therefore being a perfect match to simulate our resource management framework.

A TrueTime kernel simulates a single CPU that hosts the execution of the \RM\ and the CBS servers (virtual processors) on top of which the applications are running.
The \RM\ observes the matching function, $f_i$, of each application $i$ and computes the new reservation $v_i$ according to~(\ref{eq:RecursionForResources}). \emph{Observe} in this case means that the \RM\ is able to read the
start and finishing time of each job and it computes the matching function according to \eqref{eq:matching}. Then, it updates the virtual platforms and communicates to the applications the observations $Y_{i}(t^i_k)\equiv f_i(t^i_k)$ according to \eqref{eq:AsynchronicityCondition}.  

It is here assumed that applications are composed of some time sensitive portions of code, called \emph{jobs}. For example, in a media encoder/decoder a
job may be the encoding/decoding of an MPEG frame. Applications are
requested to inform the \RM\ about the desired duration of each
job. Below we report a template of the job code. To ease the
presentation, we omit some implementation details, which can be found
in~\cite{MagECRTS}.
%
%
\begin{lstlisting}[style=compactCStyle]{}
  /* $\mathtt{j}$ is the job index*/
  id = signal_job_start(j);
  adjust = get_performance(j);
  /* body of the job. If service aware, it should
     modify its resource requirement by $\mathtt{adjust}$ */
  do_work(/* parameters */);
  signal_job_end(j);
\end{lstlisting}
%
%
%
%
%
\begin{figure}[t]
  \begin{center}
    \iffigures
    \includegraphics[scale=1]{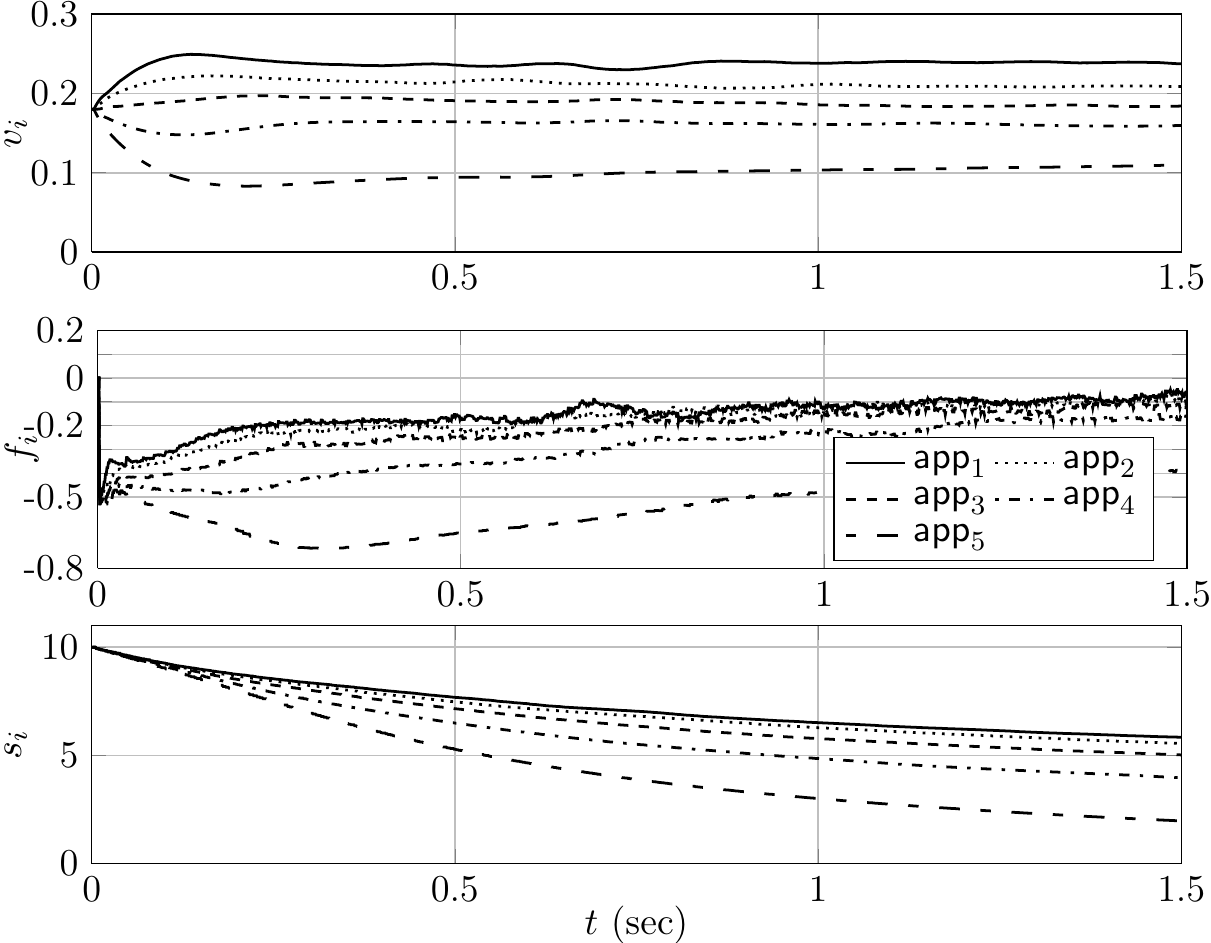}
    \fi
    \caption{TrueTime simulation of five applications with
      different weights that update their service levels.}
    \label{fig:truetimedata_slvp}
  \end{center}
\end{figure}
\begin{figure*}[t]
\centering
\begin{subfigure}[b]{\columnwidth}
  \centering
    \iffigures
    \includegraphics[scale=1]{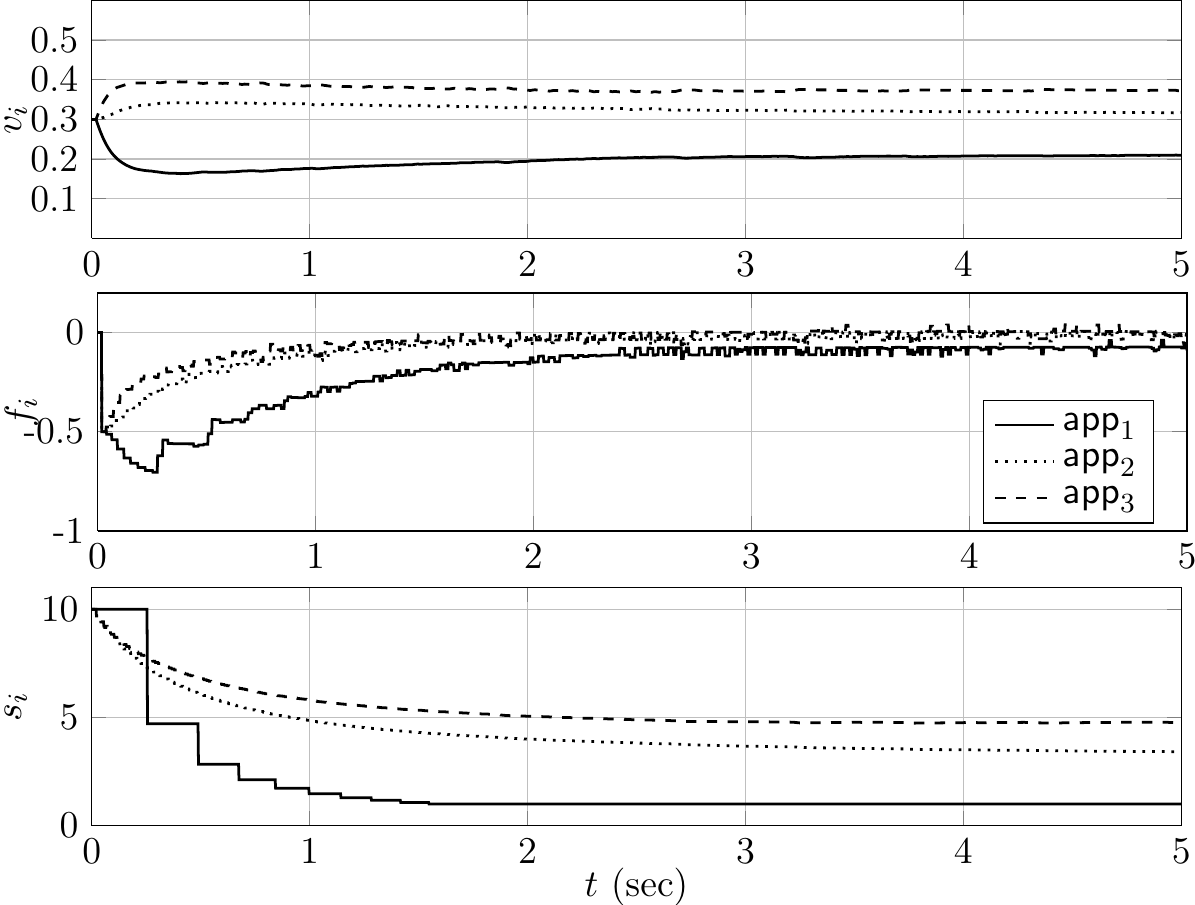}
    \fi
    \caption{\RM\ deals with asynchronicity.}
    \label{fig:truetimedata_async}
\end{subfigure}
\begin{subfigure}[b]{\columnwidth}
  \centering
    \iffigures
    \includegraphics[scale=1]{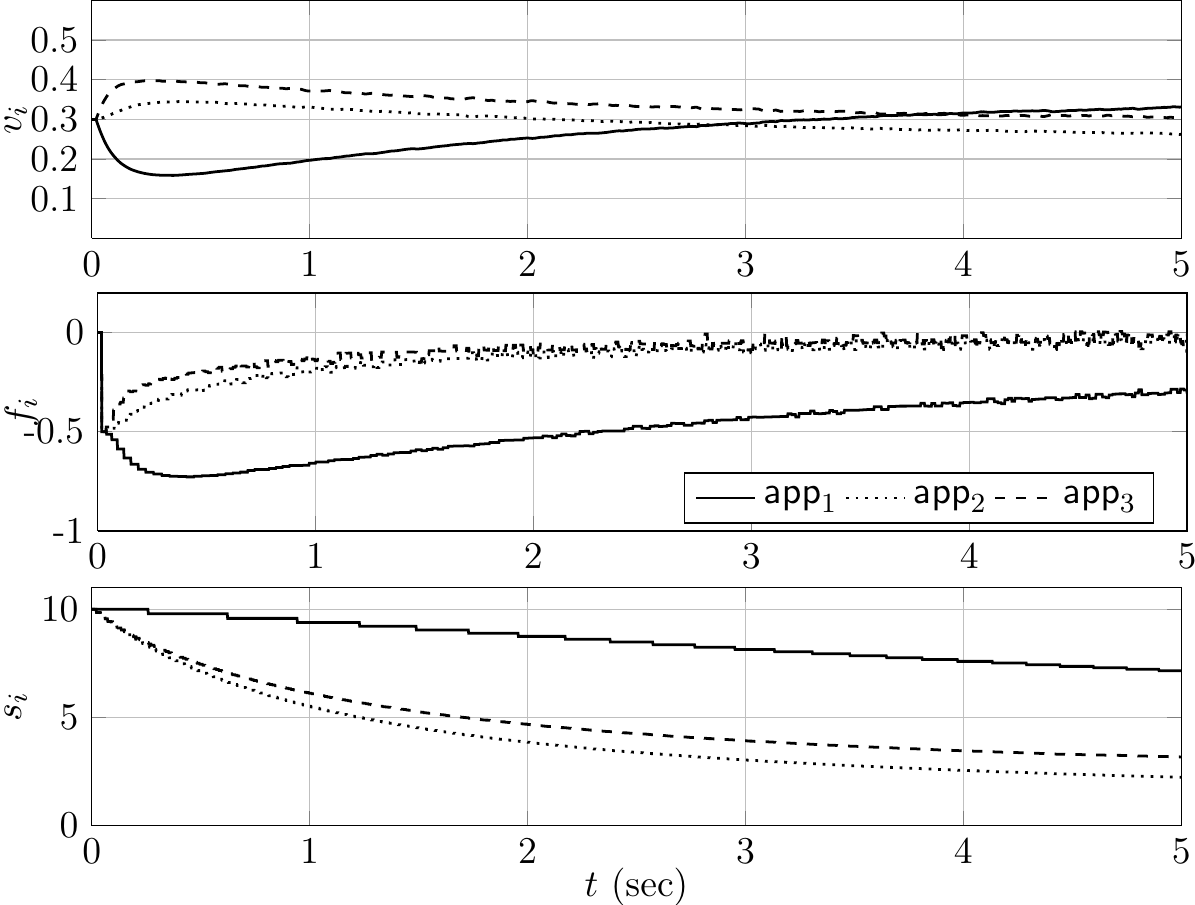}
    \fi
    \caption{\RM\ does not deal with asynchronicity.}
    \label{fig:truetimedata_async_unhandled}
\end{subfigure}%
\caption{TrueTime simulation of three applications that
      asynchronously update their service levels.}
\label{fig:asyncexp}
\end{figure*}
As the application runs, it is  asked to signal the start 
and the end of a job. This signaling actions are performed by invoking 
respectively \texttt{signal\_job\_start} and \texttt{signal\_job\_end}, 
providing as parameter the index \texttt{j} of the job. 
Within the job, the first action is the invocation of the function
\texttt{get\_performance}. This function, which is computed by the monitoring infrastructure,
returns a measurement of the service level adjustment required to
achieve a perfect matching between the service level and the virtual
platform. Jobs are assumed to be periodic.

To simulate service-aware applications, we developed a synthetic test
application, which performs some computation 
depending on the service level $s_i$. All jobs of the application have
deadline $D_i$ and are executed in a forever loop.  The execution
requirement of each job is a linear function $a_i s_i + b_i$ of the
service level.
Hence, applications with a large $a_i$ are more service-sensitive than
applications with $a_i$ close to zero.  
All applications parameters ($D_i$, $\epsilon$, $a_i$,
$b_i$), which determine the application behavior and its capacity to
adapt, are set at start time. This enables, for example, the
coexistence of fully service-aware applications together with
service-unaware ones (with $a_i=0$).

\subsection{Synchronous updates}

In the first scenario, we consider five applications running in a single core and updating synchronously with the \RM\ according to (\ref{eq:RecursionForApplications}). We pick $\lambda_{1}=0.9$, $\lambda_{2}=0.7$, $\lambda_{3}=0.5$, $\lambda_{4}=0.3$ and $\lambda_{5}=0.1$. Each application $i$ has $a_i\!\!=20$, $b_i\!\!=200$ and job deadline $D_i=1\ \text{msec}$. The initial service levels for the five applications is set to $10$ and the applications update their service levels every time they perform a new job. Finally, we restrict the maximum assignable bandwidth by the \RM\ to 90\% to simulate the fact that the operating system should have some space to execute on the same hardware. 

Figure~\ref{fig:truetimedata_slvp} reports the quantities measured during the experiment. All applications are gradually reducing their service levels as expected due to the negative matching function $f_i$. According to Theorem~\ref{Th:SynchronousConvergence}, $\mathsf{app}_1$ should receive a larger virtual platform compared to the rest of the applications due to its larger weight. The final allocation of virtual platforms may not correspond exactly to the values assigned by condition~(\ref{eq:SpecialStationaryPoint}), since the conditions partially hold at the beginning of the simulation when $f_i$ are significantly less than 0. Note though that the relative importance of the applications is preserved due to the synchronous updates.

\subsection{Asynchronous updates}

In this scenario, we investigate the effect of the asynchronicity in the final allocation of virtual platforms. We consider three different applications starting at the same time, with weights $\lambda_{1}=0.1$, $\lambda_{2}=0.5$ and $\lambda_{3}=0.8$. Each application uses resources according to $a_i\!\!=40$, $b_i\!\!=100$. Also they all have a job deadline $D_i=10\ \text{msec}$. Application $\mathsf{app}_1$ updates its service level after completing $10$ jobs, while $\mathsf{app}_2$ and $\mathsf{app}_3$ update their service levels after completing 1 job, i.e., asynchronous updates are introduced. The initial service levels of the three applications are set equal to $10$ and the upper and lower bounds on the service levels are set equal to $0$ and $20$, respectively. 

Figure~\ref{fig:truetimedata_async} reports the quantities measured during the experiment, when applications apply the scheme prescribed in the hypotheses of Theorem~\ref{Th:AsynchronousConvergence} for asynchronous updates. Each application employs a constant step-size sequence of $\epsilon = 0.03$.  It can be noticed that the \RM\ is able to maintain a virtual platform allocation that is consistent with the weights of the applications, while driving all matching functions to zero. 

To strengthen the motivation for asynchronicity management, Figure~\ref{fig:truetimedata_async_unhandled} shows the same simulation when applications do not employ the adjusted observation signal introduced in Equation~(\ref{eq:AsynchronicityCondition}), and instead they employ the originally introduced observation signal of the synchronous scheme (\ref{eq:RecursionForApplications}). Due to the slow update rate of ${\sf app}_1$, it maintains a high service level for longer period, which subsequently leads to maintaining a smaller matching function $f_1$ for longer period. Thus, even though the weight of $\mathsf{app}_1$ is smaller compared to the rest of applications, the \RM\ favors this application significantly by gradually providing more resources. This simulation demonstrates that the original scheme of Section~\ref{sec:LearningDynamics} may not be fair to applications under asynchronous updates, and it was the main motivation for the development of the updated dynamics of Theorem~\ref{Th:AsynchronousConvergence}.

%
%

\section{Conclusions}
\label{sec:conc}

We proposed a distributed management framework for allocating CPU resources to time-sensitive applications. Given that future computing systems will have to accommodate large number of applications of different demand levels, resource allocation should not be independent of the applications' performance (a notion captured through the matching function in this paper). Furthermore, given that resources are always finite, applications with higher flexibility in adjusting their demands (or service level) should decrease their service levels in overload cases. This paper proposed a distributed scheme that incorporates both these two elements, i.e., both measurements of the applications' performance, and applications' service-level adjustment. 

In this paper, service-level adjustment is performed by \textit{prescribing} certain dynamics to the applications. Such prescribed response dynamics was the first step towards the development of a fully distributed allocation scheme. In a fully distributed setup, where applications are not prescribed the response dynamics, the question is whether a \RM\ can still be designed that guarantees fair allocation of resources independently of the type of applications and their adjustment dynamics. 



\appendix

\section{Proof of Proposition~\ref{Pr:FeasibleAllocations}}	\label{Ap:FeasibleAllocations}

Let us first consider the unconstrained version of the adjustment dynamics (\ref{eq:RecursionForResources}), i.e., take $\kappa=1$. (We will revisit this assumption later.) In this case, the sum of the normalized virtual platforms can be expressed as:
\begin{eqnarray*}
 \lefteqn{\sum_{i=1}^{n}\tv_i(t_{k+1}) - 1 =}\cr && \Big(\sum_{i=1}^{n}\tv_i(t_{k}) - 1\Big)\Big(1 + \epsilon \sum_{j=1}^{n}\lambda_j\NEGs{f_j(t_k)}\Big).
\end{eqnarray*}
Given that $-1\leq \NEGs{f_i(t_k)} \leq 0$, for sufficiently small $\epsilon=\epsilon(n)>0$, the second term of the right-hand side is positive for all $k=1,2,...$. If $\sum_{i=1}^{n}\tv_i(t_k) \leq 1$, then $\sum_{i=1}^{n}\tv_i(t_{k+1}) \leq 1$. Thus, for a bounded number of applications, we may pick sufficiently small $\epsilon=\epsilon(n)$ such that $\tv_i(t_k)\in[0,1]$ and $\sum_{i=1}^{n}\tv_i(t_k) \leq 1$ for all $k=1,2,...$.
 
We consider now the constrained version of the recursion (\ref{eq:RecursionForResources}). For some time $t_k$, let us assume that $\sum_{i=1}^{n}\tv_i(t_k)\leq{1}$, i.e., the allocation is feasible. When we update this allocation using (\ref{eq:RecursionForResources}), the projection operator is activated only if $\tv_i(t_k)+\epsilon F_i(t_k) >\nicefrac{1}{\kappa}$ for some $i$. Given that $\sum_{i=1}^{n}\tv_i(t_k)+\epsilon F_i(t_k) \leq{1}$ (as we showed for the unconstrained dynamics), this quantity may only reduce after applying the projection operator. Thus, feasibility is also preserved under the constrained recursion.

\section{Proof of Proposition~\ref{Pr:StarvationAvoidance}} \label{Ap:StarvationAvoidance} 
 
First, note that $\sum_{j=1}^{n}\lambda_j\NEGs{f_j(t_k)} \geq - n,$ for all $k=0,1,...$. 
According to (\ref{eq:RecursionForResources}), the incremental difference of $\tv_i$ for the unconstrained dynamics, satisfies:
\begin{equation}	\label{eq:PreliminaryObservations:LowerBound}
\Delta\tv_i(k)\df \tv_i(t_{k+1}) - \tv_i(t_k) 
\geq \epsilon \left( -\lambda_i\NEGs{f_i(t_k)} - n\tv_i(t_k) \right).
\end{equation}
Define the set $\Gamma_i(\epsilon)\df \{\tv_i\in\cV_i: \tv_i \in (\epsilon,(L+1)\epsilon]\},$ and pick $\epsilon$ sufficiently small such that $\epsilon < \nicefrac{1}{(L+1)\kappa}$. Since $\sup_{k\in\mathbb{N}}\magn{F_i(t_k)}\leq{L}$, in order for $\tv_i(t_k)$ to approach zero, there should be a time $k^*$ at which $\tv_i(t_{k^*})\in \Gamma_i(\epsilon)$. Assuming that $\tv_i(t_k)\in \Gamma_i(\epsilon)$, if we pick $\epsilon$ sufficiently small, we have $\NEGs{f_i(t_k)}\leq \beta_i\kappa(L+1)\epsilon/\underline{s}_i - 1.$ Thus, the right-hand side of (\ref{eq:PreliminaryObservations:LowerBound}) further satisfies:
\begin{equation*}
\epsilon \left( -\lambda_i\NEGs{f_i(t_k)} - n\tv_i(t_k) \right) \geq \epsilon\lambda - (\beta_i\kappa/\underline{s}_i + n) (L+1) \epsilon^2.
\end{equation*}
For a given number of applications $n$, there exists $\epsilon^*=\epsilon^*(n)<\nicefrac{1}{(L+1)\kappa}$ with $\epsilon^*(n)\to{0}$ as $n\to\infty$, such that, if $\epsilon < \epsilon^*$, then the above quantity is strictly positive, i.e., $\Delta\tv_i(k)>0$. Since under the unconstrained dynamics the lower bound of the projection operator is not reached, the same will also hold for the constrained dynamics. From (\ref{eq:PreliminaryObservations:LowerBound}), we conclude that if $\tv_i(t_0) > \epsilon$ for all $i$, then $\inf_{k\in\mathbb{N}}\tv_i(t_k) \geq \epsilon$ $\forall{i}$.

\section{Proof of Proposition~\ref{Pr:Balance}} \label{Ap:Balance}

At time instance $k$, let $\mathcal{I}'\subseteq{\mathcal{I}}$ be the set of applications with resources greater than $\zeta$, i.e., $\mathcal{I}'\df \{i\in\mathcal{I}:\tv_i(t_k) > \zeta\}$. Pick $0<\zeta\leq\nicefrac{1}{\kappa}$ such that $\gamma^*\df\max_{\mathcal{I}\backslash\mathcal{I}'}\{\beta_i\kappa\zeta/\underline{s}_i-1\}<0$, i.e., all applications in $\mathcal{I}\backslash\mathcal{I}'$ have a negative matching function. Pick also $\epsilon<\zeta/L$.

(1) For any $i\in\mathcal{I}'$, the incremental difference of $\tv_i$ at $k$ is defined as $\Delta{\tv_i}(k)\df \tv_i(t_{k+1}) - \tv_i(t_{k}) = \epsilon F_i(t_k)$, assuming that the projection operator in (\ref{eq:RecursionForResources}) is not activated.
Note that for all $j\in\mathcal{I}\backslash\mathcal{I}'$, $\tv_j(t_k) \leq \zeta$ and
$$\sum_{j\in\mathcal{I}\backslash{i}}\lambda_j\NEGs{f_j(t_k)} \leq \Big(\sum_{j\in\mathcal{I}\backslash\mathcal{I}'}\lambda_j\Big)\gamma^* \leq \magn{\mathcal{I}\backslash\mathcal{I}'}\gamma^*\lambda.$$
Hence, according to the definition of $F_i(t_k)$, we have:
\begin{eqnarray*}
\Delta{\tv_i}(k) & \leq & -\epsilon(1-\tv_i(t_k))\lambda_i\NEGs{f_i(t_k)} +   \epsilon\magn{\mathcal{I}\backslash\mathcal{I}'}\gamma^* \tv_i(t_k)\lambda \cr
&\leq&  \epsilon(1-\zeta) + \epsilon
(n - \left\lfloor{(1-\zeta)/\zeta}\right\rfloor)\zeta\gamma^*\lambda,
\end{eqnarray*}
where the last inequality results from the fact that $-\lambda_i\NEGs{f_i(t_k)} \leq 1$, $\tv_i(t_k)>\zeta$, $1-\tv_i(t_k) < 1-\zeta$ and 
$\magn{\mathcal{I}\backslash\mathcal{I}'}\geq n - \lfloor{(1-\zeta)/\zeta}\rfloor$. For any $$n\geq n_1^*(\zeta) \df \left\lceil{\left\lfloor{\frac{(1-\zeta)}{\zeta}}\right\rfloor + \frac{-2+\zeta}{\zeta\gamma^*\lambda}}\right\rceil,$$ we have $-\zeta< -\epsilon{L}\leq\Delta{\tv_i}(k)\leq -\epsilon < 0.$ In this case, the initial assumption that the projection operator in (\ref{eq:RecursionForResources}) is not activated is also valid. Furthermore, according to \cite[Theorem~5.1]{Nevelson76}, the process $\tv_i(t_{k})$ will enter $[0,\zeta]$ within finite $k$. 

(2) For any application $i\in\mathcal{I}\backslash\mathcal{I}'$, the unconstrained incremental difference  $\Delta{\tv_i(k)}\df \epsilon F_i(t_k)$ at time $k$ satisfies:
\begin{equation*}
\Delta{\tv_i}(k) \leq \epsilon (1-\tv_i(t_k)) + \epsilon\tv_i(t_k)\gamma^*\lambda (n - \lfloor{(1-\zeta)/\zeta}\rfloor - 1),
\end{equation*}
since $-\lambda_i\NEGs{f_i(t_k)} \leq 1$, $$\sum_{j\neq{i}}\lambda_j\NEGs{f_j(t_k)}\leq\sum_{j\in\mathcal{I}\backslash\mathcal{I}'\backslash{i}}\lambda\NEGs{f_j(t_k)} \leq \gamma^*\lambda \magn{\mathcal{I}\backslash\mathcal{I}'\backslash{i}},$$ and $\magn{\mathcal{I}\backslash\mathcal{I}'\backslash{i}} \geq (n-\lfloor{(1-\zeta)/\zeta}\rfloor -1)$.
In order for the process $\tv_i(t_k)$ to exit the set $[0,\zeta]$, there should be a time instance $k^*$ at which $\tv_i(t_{k^*}) \in (\zeta-\epsilon L,\zeta]$. For any $\tv_i(t_k)\in (\zeta-\epsilon L,\zeta]$, we have:
\begin{equation*}
\Delta{\tv_i(k)} \leq \epsilon(1-\zeta+\epsilon{L}) + \epsilon(\zeta-\epsilon{L})\gamma^*\lambda(n - \lfloor{(1-\zeta)/\zeta}\rfloor - 1).
\end{equation*}
If the number of applications satisfy: $$n\geq n_2^*(\zeta) \df \left\lceil{1+\left\lfloor{\frac{1-\zeta}{\zeta}}\right\rfloor + \frac{-2+\zeta-\epsilon{L}}{(\zeta-\epsilon{L})\gamma^*\lambda} }\right\rceil,$$ then, $\Delta{\tv_i(k)}\leq -\epsilon < 0$, which implies that the unconstrained process $\{\tv_i(t_k)\}$ will not exit $[0,\zeta]$ for all future times. The same will hold for the constrained process.

Finally, by defining $n^*(\zeta)\df\max\{n_1^*,n_2^*\}$, both statements (1) and (2) will hold for any $n\geq n^*$. Note that $n^*\to\infty$ and $\zeta n^* \to c $ as $\zeta\to{0}$, for some $c>0$.

\section{Proof of Proposition~\ref{Pr:GAS}} \label{Ap:GAS}

Let $(\vs^*,\tvv^*)$ be a stationary point of the ODE~(\ref{eq:OverallODE}), where by property (\ref{eq:StationaryPoints}) satisfies $s_i^* = \underline{s}_i$. Define the function
$W(\vs,\tvv) \df \nicefrac{1}{2}(\tvv-\tvv^*)\tr(\tvv-\tvv^*) \geq 0.$
The derivative of $W$ with respect to time $\tau$ satisfies:
\begin{equation}	\label{eq:GAS:Derivative}
  \dot{W}(\vs,\tvv)  = \sum_{i=1}^{n}(\tv_i-\tv_i^*)\tr \Phi_i(\vs,\kappa\tvv).
\end{equation}
Note that at the stationary point and for every $i$, either one of the following holds: (a) $\tv_i^*=\nicefrac{1}{\kappa}$ and $\Phi_i(\vs^*,\kappa\tvv^*)>0$, or (b) $\Phi_i(\vs^*,\kappa\tvv^*)=0$. Note that case (a) implies $(\tv_i-\tv_i^*)\tr \Phi_i(\vs,\kappa\tvv) < 0$ for all $\tv_i<\nicefrac{1}{\kappa}$ and all $s_i>\underline{s}_i$. Thus, in case condition (a) is satisfied for some applications, the corresponding additive terms in (\ref{eq:GAS:Derivative}) will be strictly negative. Without loss of generality, it suffices to investigate the derivative (\ref{eq:GAS:Derivative}) when all applications satisfy condition (b). Furthermore, it suffices to consider the case where $s_i\equiv s_i^*=\underline{s}_i$, since under hypothesis (i) $\phi_i(s_i,v_i) < 0$ for all $s_i\in\mathcal{S}_i$, $v_i\in\cV_i$ and $i$, and under hypothesis (ii) there exists a time $\tau^*$ after which $\phi_i(s_i(\tau),\kappa\tv_i(\tau))<0$, for all $\tau\geq{\tau}^*$ (according to Proposition~\ref{Pr:Balance} and convergence of Euler's method shown in Proposition~\ref{Pr:LAS}). Thus, for all $\tau\geq{\tau^*}$, we have:
\begin{eqnarray*}
  \lefteqn{\dot{W}(\vs^*,\tvv)}\cr & = & \sum_{i=1}^{n}(\tv_i-\tv_i^*) \left(\Phi_i(\vs^*,\kappa\tvv) - \Phi_i(\vs^*,\kappa\tvv^*)\right) \cr 
  & = & -\Big(\sum_{j=1}^{n}\lambda_j\Big)\sum_{i=1}^{n}|\tv_i-\tv_i^*|^2 -\sum_{i=1}^{n}\frac{\lambda_i\beta_i\kappa}{s_i^*}|\tv_i-\tv_i^*|^2 \cr
  && -\sum_{i=1}^{n}(\tv_i-\tv_i^*)\sum_{j=1}^{n}\frac{\lambda_j\beta_j\kappa}{s_j^*} \left(\tv_i^*\tv_j^* - \tv_i\tv_j\right).
\end{eqnarray*}
We denote by $I_1$, $I_2$ and $I_3$ the three terms of the r.h.s. of the above expression, i.e., $\dot{W}(\vs^*,\tvv)\equiv I_1+I_2+I_3$. Note that:
$I_1 = - (\sum_{j=1}^{n}\lambda_j)\left\|\tvv - \tvv^*\right\|_2^2$,
$\magn{I_2} \leq \max_{i\in\mathcal{I}}\{\nicefrac{\beta_i\kappa}{s_i^*}\}\|\tvv-\tvv^*\|_2^2$
and
\begin{eqnarray*}
\magn{I_3} & \leq & \sum_{i=1}^{n}|\tv_i-\tv_i^*|\sum_{j=1}^{n}\frac{\lambda_j\beta_j\kappa}{s_j^*}\left|\tv_i\tv_j - \tv_i^*\tv_j^* \right| \cr 
& \leq & \sum_{i=1}^{n}|\tv_i-\tv_i^*|\sum_{j=1}^{n}\frac{\lambda_j\beta_j\kappa}{s_j^*}\Big(\left|\tv_j-\tv_j^*\right| \tv_i + \left|\tv_i-\tv_i^*\right| \tv_j^* \Big) \cr 
& \leq & \sup_{i,\tau\geq{\tau^*}}\{\tv_i\}\max_{j\in\mathcal{I}}\Big\{\frac{\lambda_j\beta_j\kappa}{s_j^*}\Big\}\|\tvv-\tvv^*\|_1^2 + \cr &&
\max_{j\in\mathcal{I}}\Big\{\frac{\lambda_j\beta_j\kappa}{s_j^*}\Big\}\sum_{j=1}^{n}\tv_j^* \|\tvv-\tvv^*\|_2^2.
\end{eqnarray*}
Given that $\|\tvv-\tvv^*\|_1 \leq \sqrt{n}\|\tvv-\tvv^*\|_2$, $\sum_{j=1}^{n}\tv_j^*\leq{1}$ and $\lambda_i \le 1$ for all $i$, we have
\begin{eqnarray*}
\magn{I_3} \leq \max_{j\in\mathcal{I}}\Big\{\frac{\beta_j\kappa}{s_j^*}\Big\} \left( n \sup_{i,\tau\geq{\tau^*}}\{\tv_i\} + 1 \right) \|\tvv-\tvv^*\|_2^2.
\end{eqnarray*}
Under hypothesis (i), i.e., as $\beta_i/\underline{s}_i\to{0}$ for all $i$, and for some fixed size of applications $n$, the first term, $I_1$, dominates in size the term $I_2+I_3$ uniformly in time. Since $I_1<0$ for any $\tvv\neq\tvv^*$, we have that $\dot{W}(\vs^*,\tvv)<0$ for any $\tvv\neq\tvv^*$. Under hypothesis (ii), i.e., as $n\to\infty$, Proposition~\ref{Pr:Balance} implies that $n\sup_{i,\tau\geq{\tau^*}}\{\tv_i(\tau)\}\le \zeta n(\zeta)$ approaches a positive constant. Thus, the first term $I_1$ dominates in size the term $I_2+I_3$ when $n\to\infty$. In this case, $\dot{W}(\vs^*,\tvv)<0$ for any $\tvv\neq\tvv^*$ such that $\tv_i \leq \zeta$ for all $i$. Thus, under either (i) or (ii), and by \cite[Theorem~3.2]{Khalil92}, we conclude that the unique stationary point $(\vs^*,\tvv^*)$, satisfying (\ref{eq:SpecialStationaryPoint}), is globally asymptotically stable.

\section{Proof of Proposition~\ref{Pr:FictitiousEquivalence}} \label{Ap:FictitiousEquivalence}

(1) Let us first consider the unconstrained versions of the actual (\ref{eq:ActualRecursion}) and fictitious update (\ref{eq:FictitiousUpdate}). In this case, the corresponding linear-time interpolations satisfy for any run time $t>0$:
 \begin{eqnarray*}
   \lefteqn{s_{i,\epsilon}(t) - s_{i,\epsilon}'(t)} \cr 
   & = & s_i(t_{\bar{k}(t,i)}^i) - s_i'(t_{\bar{m}(t)}) \cr
   & = & \sum_{k=0}^{\bar{k}(t,i)-1}\epsilon Y_{i}(t_{\bar{m}(t_{k}^i)}) - \sum_{m=0}^{\bar{m}(t)-1}\epsilon Y_i'(t_{\psi_i(m)}) \cr 
   & = & \sum_{k=0}^{\bar{k}(t,i)-1}\epsilon Y_{i}(t_{\bar{m}(t_{k}^{i})}) - \sum_{m=\bar{m}(t_{\bar{k}(t,i)-1}^i)+1}^{\bar{m}(t)-1}\epsilon Y_i'(t_{\psi_i(m)}) \cr &&
   -\sum_{k=0}^{\bar{k}(t,i)-2}\epsilon N_i(k) Y_i'(t_{\psi_i(\bar{m}(t_{k+1}^i))}), \cr &&
 \end{eqnarray*}
where the third summation of the r.h.s. summarizes all the observation terms of the \RM\ up to time index $\bar{m}(t_{\bar{k}(t,i)-1}^{i})$, and the second summation summarizes all the remaining terms, i.e., from $\bar{m}(t_{\bar{k}(t,i)-1}^i)+1$ until $\bar{m}(t)-1$. For example, in Figure~\ref{fig:Asynchronicity}, the third summation corresponds to all time indices up to $\bar{m}(t_{\bar{k}-1}^{i})$, while the second summation corresponds to the remaining terms up to time $\bar{m}(t)-1$. Since $1\leq N_i(k)\leq\bar{N}$, the latter terms may be zero to maximum $\bar{N}$ in numbers. Also, $\psi_i(\bar{m}(t_{k+1}^i)) = \bar{m}(t_k^i)$. Thus, we have: 
 \begin{eqnarray*}
  \lefteqn{s_{i,\epsilon}(t) - s_{i,\epsilon}'(t) = } \cr 
  && \sum_{k=0}^{\bar{k}(t,i)-2}\epsilon \left(Y_{i}(t_{\bar{m}(t_k^i)}) - N_i(k) Y_i'(t_{\bar{m}(t_k^i)}) \right) + \cr &&
   \epsilon Y_{i}(t_{\bar{m}(t_{\bar{k}(t,i)-1}^{i})}) - \sum_{m=\bar{m}(t_{\bar{k}(t,i)-1}^i)+1}^{\bar{m}(t)-1}\epsilon Y_i'(t_{\psi_i(m)}).
 \end{eqnarray*}
 From the last expression, we observe that if 
 $Y_{i}(t_{k}^i) = N_i(k) Y_i'(t_{k}^i),$ then, the first term in the r.h.s. becomes identically zero. Given that $\sup_{i,t\ge{0}}\magn{Y_i'(t)} \le \ell$, for some $\ell>0$, we also have $\sup_{i,t\geq{0}}\magn{Y_{i}(t)}\leq \ell \bar{N}$. Hence, 
 \begin{eqnarray*}
  \vert s_{i,\epsilon}(t) - s_{i,\epsilon}'(t)\vert & \leq & \epsilon \ell\bar{N} + \epsilon\Big(\bar{m}(t)-\bar{m}(t_{\bar{k}(t,i)-1}^{i})-1\Big)\ell \cr & \leq & \epsilon \ell (3\bar{N} - 1),
 \end{eqnarray*}
 which approaches zero as $\epsilon\to{0}$ uniformly in time. Thus, we showed that $s_{i,\epsilon}(\cdot)$ and $s_{i,\epsilon}'(\cdot)$ are equivalent according to Definition~\ref{def:EquivalentUpdates}. Since the corresponding projected versions are simply truncations to the set $\mathcal{S}_i\equiv[\underline{s}_i,\infty)$, the same conclusion applies for the projected versions.

(2) Let $\{s_i''(t_k^i)\}_k$ denote the service level recursion under the synchronous update (\ref{eq:RecursionForApplications}) in order to distinguish it from the actual asynchronous one (\ref{eq:ActualRecursion}). Let us also denote $s_{i,\epsilon}''(\cdot)$ the corresponding linear-time interpolation. Similarly to the proof of part (1), it suffices to consider the unconstrained recursions. For every $i$ and $t\geq{0}$, and since $1\leq N_i(k)\leq\bar{N}$, we have:
\begin{eqnarray*}
 \vert s_{i,\epsilon}'(t) - s_{i,\epsilon}''(t) \vert & = & \vert s_i'(t_{\psi_i(\bar{m}(t))}) - s_i''(t_{\bar{m}(t)}) \vert \cr 
 & \leq & \sum_{k=\psi_i(\bar{m}(t))}^{\bar{m}(t)-1}\epsilon \vert Y_{i}'(t_k) \vert \cr
 & \leq & \epsilon (\bar{m}(t) - \psi_i(\bar{m}(t)) \ell \leq \epsilon \bar{N} \ell,
\end{eqnarray*}
which implies equivalence of $s_{i,\epsilon}'(\cdot)$ and $s_{i,\epsilon}''(\cdot)$.


\bibliographystyle{abbrv} \bibliography{automatica-gtrm_new}

\end{document}